\documentclass[12pt,a4paper]{amsart}
\usepackage{mathrsfs,color}
\usepackage{amsfonts}

\usepackage[active]{srcltx}
\usepackage{enumerate}
\usepackage[parfill]{parskip}

\usepackage{amsmath,amssymb,xspace,amsthm}
\usepackage[all,2cell]{xy}

\newtheorem{theorem}{Theorem}
\newtheorem{lemma}[theorem]{Lemma}
\newtheorem{remark}[theorem]{Remark}
\newtheorem{corollary}[theorem]{Corollary}

\newtheorem{example}[theorem]{Example}

\newtheorem{problem}[theorem]{Problem}

\newcommand{\tto}{\twoheadrightarrow}

\title[Graded simple Lie algebras and representations]{\bf Graded simple Lie algebras
and\\ graded simple representations}
\author{Volodymyr Mazorchuk and Kaiming Zhao}

\date{}

\begin{document}

\begin{abstract}
Let $Q$ be an abelian group  and $\Bbbk$ a field. We prove that 
any $Q$-graded simple Lie algebra $\mathfrak{g}$ over $\Bbbk$ 
is isomorphic to a loop algebra in case $\Bbbk$ has a primitive 
root of unity of order $|Q|$, if $Q$ is finite, or 
$\Bbbk$ is algebraically closed and $\dim \mathfrak{g}<|\Bbbk|$.

For  $Q$-graded simple modules over any $Q$-graded Lie algebra 
$\mathfrak{g}$, we propose a similar construction of all 
$Q$-graded simple modules over any $Q$-graded Lie algebra 
over $\Bbbk$   starting from nonextendable 
gradings of simple $\mathfrak{g}$-modules. We prove that any $Q$-graded simple 
module over  $\mathfrak{g}$ is isomorphic to a    loop module in case
$\Bbbk$ has a primitive root of unity of order $|Q|$, if $Q$ is 
finite,  or $\Bbbk$ is algebraically closed and 
$\dim \mathfrak{g}<|\Bbbk|$.
The isomorphism problem for simple graded modules constructed 
in this way remains open. For finite-dimensional $Q$-graded 
semisimple algebras we obtain a graded analogue of the Weyl Theorem.
\vspace{5mm}


\end{abstract}
\maketitle


\section{Introduction}\label{s0}

\subsection{General overview}\label{s0.1}

The present paper studies $Q$-graded simple Lie algebras 
for any abelian group $Q$ and $Q$-graded simple modules 
over $Q$-graded  Lie algebras over a field $\Bbbk$ under 
some mild restrictions.

Study of gradings on Lie algebras goes back at least as far as to the paper \cite{PZ}
which started a systematic approach to understanding of gradings by abelian
groups on simple finite dimensional Lie algebras over algebraically
closed fields of characteristic $0$. In the past two decades, there
was a significant interest to the study of  gradings on simple Lie algebras by
arbitrary groups, see the recent monograph \cite{EK2} and references
therein. In particular, there is an essentially complete classification of
fine gradings (up to equivalence) on all finite-dimensional simple Lie algebras over an
algebraically closed field of characteristic $0$, see \cite{EK2,E,Y}.
Some properties of simple $\mathbb{Z}_2$-graded Lie algebras were studied in \cite{Z}.
A classification of various classes of $\mathbb{Z}^2$-graded  Lie algebras with
at most one-dimensional homogeneous spaces was obtained in \cite{OZ1, OZ2, OZ3}.
A classification of simple Lie algebras with a $\mathbb{Z}^n$-grading such that
all homogeneous spaces are one-dimensional was obtained in \cite{KM}. For a given abelian group $Q$, a
classification of $Q$-gradings (up to isomorphism) on classical
simple Lie algebras over an algebraically closed field of characteristic different 
from $2$ was obtained in \cite{BK,E}, see also \cite{EK2}. In \cite{ABFP}, one finds
some characterizations of graded-central-simple algebras with split centroid, 
see Correspondence Theorem~7.1.1 in \cite{ABFP}. In general, it is difficult 
to find the centroid for a graded simple algebra. In the paper we will establish 
a correspondence theorem for graded-simple  Lie algebras without assuming they are 
central or with  split centroid.

When $Q$ is finite, Billig and Lau described quasi-finite $Q$-graded-simple 
modules over $Q$-graded associative (or Lie) algebras in \cite{BL}.
With respect to this, we will study $Q$-graded-simple modules over $Q$-graded 
Lie algebras without assuming qasifiniteness of modules or finiteness of $Q$.

\subsection{Notation and setup}\label{s0.2}

Throughout this paper, $\Bbbk$ denotes an arbitrary field with some restrictions in the context.
If not explicitly stated otherwise,
all vector spaces, algebras and tensor products are assumed to be over $\Bbbk$.
As usual, we denote by $\mathbb{Z}$, $\mathbb{N}$, $\mathbb{Z}_+$ and
$\mathbb{C}$ the sets of integers, positive integers, nonnegative integers
and complex numbers, respectively.

Let $Q$ be an additive abelian group. A {\em $Q$-graded Lie algebra
over $\Bbbk$} is a Lie algebra $\mathfrak{g}$ over $\Bbbk$ endowed
with a decomposition
\begin{displaymath}
 \mathfrak{g}=\bigoplus_{\alpha\in
Q}\mathfrak{g}_{\alpha}\,\, {\text{ such that }}\,\,
[\mathfrak{g}_{\alpha},\mathfrak{g}_{\beta}]\subset
\mathfrak{g}_{\alpha+\beta}\,\,  {\text{ for all }}\,\,
\alpha,\beta\in Q.
\end{displaymath}
Recall that a graded Lie algebra is called {\it graded simple} if
$\mathfrak{g}$ is not commutative and does not contain any
non-trivial graded ideal. Without loss of generality, for a  $Q$-graded simple Lie algebra
$\mathfrak{g}$ we will always assume that the elements $\alpha\in Q$ with
$\mathfrak{g}_\alpha\ne0$ generate $Q$. Gradings satisfying this condition
will be called {\em minimal}.

Let  $Q'$ be another  abelian group, $\mathfrak{g} $  a
$Q$-graded Lie algebra over a field $\Bbbk$  and $\mathfrak{g}'$ be
a $Q'$-graded Lie algebra over $\Bbbk$. We say that graded
Lie algebras $\mathfrak{g}$ and $ \mathfrak{g}'$ are {\it graded
isomorphic} if there is a group isomorphism $\tau:Q\to Q'$ and a Lie
algebra isomorphism $\sigma: \mathfrak{g}\to \mathfrak{g}'$ such
that $\sigma(\mathfrak{g}_\alpha)=\mathfrak{g}'_{\tau(\alpha)}$ for
all $\alpha\in Q.$

 A  {\em $Q$-graded module} $V$ over  a $Q$-graded
Lie algebra $\mathfrak{g}$ is a $\mathfrak{g}$-module endowed with a
decomposition
\begin{displaymath}
V=\bigoplus_{\alpha\in Q}V_{\alpha}\,\, {\text{ such that }}\,\,
\mathfrak{g}_{\alpha}\cdot V_{\beta}\subset V_{\alpha+\beta}\,\,
{\text{ for all }}\,\,\alpha,\beta\in Q.
\end{displaymath}
A graded module $V$ is called {\it graded simple} if $\mathfrak{g}V\ne0$
and $V$ does not contain any non-trivial graded submodule.

Let   $\mathfrak{g} $ be  a $Q$-graded Lie algebra over
$\Bbbk$  and
\begin{displaymath}
W'=\bigoplus_{\alpha\in Q}W_{\alpha},\,\,\,\,\,\, W=\bigoplus_{\alpha\in
Q}W'_{\alpha}
\end{displaymath}
be two $Q$-graded $\mathfrak{g}$-module. A $\mathfrak{g}$-module isomorphism
$\sigma: W\to W'$ is called a {\it graded isomorphism of degree $\beta\in Q$}
if $\sigma(W_\alpha)=W'_{\alpha+\beta}$ for all $\alpha\in Q$.
We say that $W$ and $W'$  are {\it graded
isomorphic} if there is a graded isomorphism $\sigma: W\to W'$ of some degree $\beta\in Q$.

\subsection{Results and structure of the paper}\label{s0.3}

In Section~\ref{s1},  for a subgroup $P$ of  an abelian group  $Q$ and  
a simple Lie algebra $\mathfrak{a}$ with a fixed $Q/P$-grading, 
we recall the  loop algebras which are  $Q$-graded simple Lie
algebras $\mathfrak{g}(Q,P,\mathfrak{a})$. If $P$
is  finite and $ \Bbbk$ contains a primitive root of unity of order $|P|$, then
the algebra $\mathfrak{g}(Q,P,\mathfrak{a})$ is a direct sum of
$|P|$ ideals that are isomorphic to $\mathfrak{a}$. If
$\mathrm{char}(\Bbbk)$ does divide $|P|$, then the algebra
$\mathfrak{g}(Q,P,\mathfrak{a})$ is not semisimple as an ungraded algebra.

In Section~\ref{s9}, we assume that $Q$ is  finite and $ \Bbbk$ contains 
a primitive root of unity of order $|P|$. We prove that
any $Q$-graded simple Lie algebra  has to be of the form
$\mathfrak{g}(Q,P,\mathfrak{a})$ for some simple Lie algebra
$\mathfrak{a}$ with a $Q/P$-grading, see Corollary~\ref{Finite2}
(we actually prove a more general result in Theorem~\ref{Finite}).
Our proof is based on a detailed analysis of properties of the character 
group of $Q$ and is directed towards showing existence of a non-graded 
simple ideal in the case when the underlying ungraded Lie algebra is not simple. 
Thanks to the recent classification of all gradings on finite
dimensional simple Lie algebras, see \cite{EK2, E, Y}, we thus 
obtain a full classification of all finite-dimensional $Q$-graded
simple Lie algebras  over any algebraically closed field of
characteristic $0$.

In Section~\ref{s25} we establish a graded analogue of Schur's Lemma. 
It is frequently used in the remainder of the paper. In Section~\ref{s1.3},
for arbitrary additive abelian group $Q$, using Correspondence Theorem~7.1.1 
from \cite{ABFP}, we prove that any $Q$-graded simple Lie algebra 
$\mathfrak{g}$ over an algebraically closed field $\Bbbk$ with $\dim{\mathfrak{g}}<|\Bbbk|$
has to be of the form $\mathfrak{g}(Q,P,\mathfrak{a})$, for some
simple Lie algebra $\mathfrak{a}$ with a $Q/P$-grading, see
Theorem~\ref{Infinite1}.

In Section~\ref{s2}, using our classification of
$Q$-graded simple Lie algebras, we prove a graded analogue of the Weyl's Theorem,
see Theorem~\ref{Weyl}.
Namely, we show that  any $Q$-graded finite dimensional module over a
$Q$-graded semi-simple finite dimensional Lie algebra over an algebraically
closed field of characteristic $0$ is completely graded-reducible.

Finally, in the last section of the paper, we reduce classification of
all $Q$-graded simple modules $W$ over any $Q$-graded Lie algebras 
$\mathfrak{g}$ over  $\Bbbk$ to the study of nonextendable gradings on simple
$\mathfrak{g}$-modules. For any simple $\mathfrak{g}$-module $V$ with a 
$Q/P$-grading, we first construct a $Q$-graded $\mathfrak{g}$-module 
$M(Q, P, V)$ which are called loop modules. We prove that any $Q$-graded 
simple module over $ \mathfrak{g}$    is isomorphic to a   loop module if
$\Bbbk$ has a primitive root of unity of order $|Q|$ in the case of finite 
$Q$; or $\Bbbk$ is algebraically closed with $\dim \mathfrak{g}<|\Bbbk|$, 
see Theorems~\ref{FM0} and \ref{Mod2}. We finish the paper with an 
open problem: find necessary and sufficient conditions 
for two graded simple $\mathfrak{g}$-modules $M(Q,P,V)$ and $M(Q,P',V')$ 
to be isomorphic.

\section{Construction of graded  simple Lie algebras}\label{s1}

\subsection{Construction}\label{s1.1}
We will recall some concepts and results from \cite{ABFP} with different 
notation and discuss some properties on  graded simple Lie algebras.
In this section $\Bbbk$ is an arbitrary field.

Let $Q$ be an abelian group and  $P$ a
subgroup of $Q$. Assume we are given a simple Lie algebra
$\mathfrak{a}$ over $\Bbbk$ with a fixed $Q/P$-grading
\begin{displaymath}
\mathfrak{a}=\displaystyle\bigoplus_{\bar\alpha\in
Q/P}\mathfrak{a}_{\bar\alpha}.
\end{displaymath}
Consider the group algebra $\Bbbk Q$ with the standard basis
$\{t^{\alpha}\,:\,\alpha\in Q\}$ and multiplication $t^{\alpha}
t^{\beta}=t^{\alpha+\beta}$ for all $\alpha,\beta\in Q$.
Then we can form the Lie algebra ${\mathfrak{a}}\otimes \Bbbk Q$.
For $x,y\in \mathfrak{a}$ and $\alpha,\beta\in Q$, we have
\begin{displaymath}
[x\otimes t^{\alpha},y\otimes t^{\beta}]=[x,y]\otimes t^{\alpha+\beta}.
\end{displaymath}
Define the  $Q$-graded Lie algebra
\begin{displaymath}
\mathfrak{g}(Q,P,\mathfrak{a}):= \bigoplus_{\alpha\in
Q}\mathfrak{g}(Q,P,\mathfrak{a})_{\alpha},\quad \text{ where }\quad
\mathfrak{g}(Q,P,\mathfrak{a})_{\alpha}:=\mathfrak{a}_{\bar\alpha}\otimes t^{\alpha}.
\end{displaymath}
For example, $\mathfrak{g}(Q,Q,\mathfrak{a})=\mathfrak{a}\otimes
\Bbbk Q$ (with the obvious $Q$-grading) while
$\mathfrak{g}(Q,\{0\},\mathfrak{a})=\mathfrak{a}$ (with the original
$Q$-grading). These graded Lie algebras $\mathfrak{g}(Q,P,\mathfrak{a})$ 
are called loop algebras in \cite{ABFP}. From the definition it follows that
$\dim\mathfrak{g}(Q,P,\mathfrak{a})=\dim(\mathfrak{a})|P|$ if
$\mathfrak{a}$ is finite-dimensional and $P$ is finite. The following 
result is Lemma~5.1.1 in \cite{ABFP}.

\begin{lemma}\label{Property1}
If $Q$ is an abelian group, $P$ a subgroup of $Q$  and $\mathfrak{a}$ 
is a simple Lie algebra with a $Q/P$-grading, then the algebra
$\mathfrak{g}(Q,P,\mathfrak{a})$ is a $Q$-graded simple Lie algebra.
\end{lemma}

Making a parallel with
affine Kac-Moody algebras \cite{K, MP}, it is natural to divide
these algebras into two classes. The algebras
$\mathfrak{g}(Q,Q,\mathfrak{a})$ will be called {\it untwisted
graded simple Lie algebra} while all other algebras will be called
{\it twisted graded simple Lie algebras}.

\subsection{Properties of $\mathfrak{g}(Q,P,\mathfrak{a})$ in the case of finite $Q$}\label{s1.17}

Now we need to establish some properties of the graded simple Lie
algebras $\mathfrak{g}(Q,P,\mathfrak{a})$ for finite groups $Q$. So
in the rest of this section we assume that $Q$ is finite and that
$\Bbbk$ contains a primitive root of unity of order $|Q|$ (which implies 
that $\mathrm{char}(\Bbbk)$ does not divide $|Q|$).

Let $\hat{Q}$ denote the character  group of $Q$, that is the group
of all group homomorphisms $Q\to \Bbbk^{*}$ under the operation of
pointwise multiplication. Note that $\hat{Q}\cong Q$ because of our
assumption on $\mathrm{char}(\Bbbk)$. For any $f\in\hat{Q}$, we
define the associative algebra automorphism
\begin{displaymath}
\tau_{f}:\Bbbk Q\to \Bbbk Q\quad\text{ via }\quad
\tau_f(t^\alpha)=f(\alpha)t^\alpha\quad \text{ for all }\alpha\in Q.
\end{displaymath}
This induces the Lie algebra automorphism
\begin{equation}\label{eq71}
\begin{array}{rrcl}
\tau_f:&\mathfrak{g}(Q,P,\mathfrak{a})&\to& \mathfrak{g}(Q,P,\mathfrak{a})\\
&x_{\bar{\alpha}}\otimes t^\alpha&\mapsto&f(\alpha)x_{\bar{\alpha}}\otimes t^\alpha
\end{array}
\end{equation}
for all $\alpha\in Q$ and $x_{\bar\alpha}\in
{\mathfrak{a}}_{\bar\alpha}$. Note that $\tau_{fg}=\tau_f\tau_g$ for
all $f,g\in\hat{Q}$, in other words, $\hat{Q}$ acts on
$\mathfrak{g}(Q,P,\mathfrak{a})$ via automorphisms $\tau_f$. We will
use the following:

\begin{remark}\label{remabgroup}
{\rm If $Q$ is a finite abelian group, $P$ a subgroup of $Q$,
$\hat{Q}$ the group of characters of $Q$ over a field $\Bbbk$ such
that $\mathrm{char}(\Bbbk)$ does not divide $|Q|$ and
$P^{\perp}:=\{f\in \hat{Q}\,:\, f(\alpha)=1\text{ for all }\alpha\in
P\}$, then $|\hat{Q}/P^{\perp}|=|P|$. Indeed, because of our
assumption on $\Bbbk$, we know that $|\hat{P}|=|P|$. Therefore it is
enough to prove that each character of $P$ can be extended to a
character of $Q$. The latter follows directly from the Frobenius
reciprocity. }
\end{remark}

\begin{lemma}\label{Property19}
Assume that  $Q$ is  finite and $\Bbbk$ contains a primitive  root of unity of order $|Q|$.
Then the algebra $\mathfrak{g}(Q,P,\mathfrak{a})$ is a direct sum of
$|P|$ ideals. Each of these ideals is $Q/P$-graded and, moreover,
isomorphic to $\mathfrak{a}$ as $Q/P$-graded Lie algebras.
\end{lemma}

\begin{proof}
For $\alpha\in Q$, set
\begin{displaymath}
\underline{t}^{\alpha}:=t^{\alpha}\sum_{\beta\in P}t^{\beta}.
\end{displaymath}
Then, for any $\alpha,\alpha'\in Q$, we have
$\underline{t}^{\alpha}=\underline{t}^{\alpha'}$ if and only if
$\alpha-\alpha'\in P$. Consider the vector space
\begin{equation}
I=\displaystyle\bigoplus_{\bar\alpha\in
Q/P}{\mathfrak{a}}_{\bar\alpha}\otimes \underline{t}^\alpha,
\end{equation}
which is well-defined because of the observation in the previous
sentence. Since
$\underline{t}^{\alpha}{t}^{\beta}=\underline{t}^{\alpha+\beta}$, for
all $\alpha,\beta\in Q$, the space $I$ is an ideal. Note that
$[I,I]\neq 0$ since $\mathfrak{a}$ is a simple Lie algebra and
$\mathrm{char}(\Bbbk)$ does not divide $|Q|$ (and thus it does not
divide $|P|$ either, which implies
$\underline{t}^{\alpha}\underline{t}^{\alpha}=|P|\underline{t}^{\alpha}\neq
0$). It follows that $I\cong \mathfrak{a}$ as a $Q/P$-graded Lie
algebra. Simplicity of $\mathfrak{a}$ even implies that $I$ is a
minimal ideal.

Define the {\it invariant subgroup $\mathrm{Inv}(I)$ of $I$} as
\begin{displaymath}
\mathrm{Inv}(I)=\{f\in \hat{Q}\,|\,\tau_f(I)=I\} ,
\end{displaymath}
which is a subgroup of $\hat{Q}$. Then the set
$\mathbf{I}:=\{\tau^{\alpha}(I)\,:\, \alpha\in \hat{Q}\}$ consist of
$|\hat{Q}/\mathrm{Inv}(I)|$ different minimal ideals of
$\mathfrak{g}(Q,P,\mathfrak{a})$. From the definitions it follows
that
\begin{displaymath}
\mathrm{Inv}(I)\subset P^{\perp}:=\{f\in \hat{Q}\,\vert\,
f(\alpha)=1\text{ for all }\alpha\in P\}
\end{displaymath}
and hence $|\hat{Q}/\mathrm{Inv}(I)|\geq|\hat{Q}/P^{\perp}| =|P|$,
see Remark~\ref{remabgroup} for the latter equality. Comparing the
number of non-zero homogeneous components in
$\mathfrak{g}(Q,P,\mathfrak{a})$ and in the subspace
\begin{displaymath}
\bigoplus _{J\in\mathbf{I}}J\subset \mathfrak{g}(Q,P,\mathfrak{a}),
\end{displaymath}
we deduce that these two algebras coincide. The statement of the lemma follows.
\end{proof}

The ideals in Lemma \ref{Property19} are the only ones of the Lie algebra $\mathfrak{g}$.

\begin{example}\label{example0}
{\rm In case  $Q$ is  finite and $\mathrm{char}(\Bbbk)$ does divide
$|Q|$, the algebra $\mathfrak{g}(Q,P,\mathfrak{a})$ is not a direct
sum of simple ideals in general. For example, let us consider the
case that $\mathrm{char}(\Bbbk)=|Q|$, $Q=\mathbb{Z}_p$ and $P=0$.
Since $(t^{\overline{1}})^p-1=(t^{\overline{1}}-1)^p$, we have that
\begin{displaymath}
\mathfrak{g}(Q,P,\mathfrak{a})\simeq
\mathfrak{a}\otimes \left(\mathbb{C}[x]/\langle x^p\rangle\right),
\end{displaymath}
where $x=t^{\overline{1}}-1$.
The latter algebra has an abelian ideal $\mathfrak{a}\otimes x^{p-1}$ and a
nilpotent ideal $\mathfrak{a}\otimes x\mathbb{C}[x]$. The latter is,
in fact, a maximal ideal.
}
\end{example}


\section{Graded  simple Lie algebras: the case of finite $Q$}\label{s9}

\subsection{Preliminaries}\label{s1.2}

In this section, for a finite abelian group $Q$, we will characterize all
$Q$-graded simple Lie algebras $\mathfrak{g}$ (without assuming that 
$\mathfrak{g}$ is central or with split centroid)  
when $\Bbbk$ is an arbitrary field containing  a primitive root of 
unity of order $|Q|$. That is, in this setup we establish an analogue 
of Correspondence Theorem~7.1.1 from \cite{ABFP}. We start with an arbitrary additive abelian group $Q$ at this moment, 
not necessarily finite,  and a $Q$-graded simple Lie algebra 
$$\mathfrak{g}=\displaystyle\bigoplus_{\alpha\in Q}\mathfrak{g}_\alpha.$$

Every $x\in \mathfrak{g}$ can be written in the form
$x=\displaystyle\sum_{\alpha\in Q}x_\alpha$
where $x_\alpha\in\mathfrak{g}_\alpha$.
In what follows the notation $x_\alpha$, for some $\alpha\in Q$,
always means $x_\alpha\in\mathfrak{g}_\alpha$.
We define the {\it support}
of $x$ as
\begin{displaymath}
\mathrm{supp}(x):=\{\alpha\in Q\,|\,x_\alpha\ne0\}.
\end{displaymath}
Similarly, we can define $\mathrm{supp}(X)$ for any nonempty subset
$X\subset \mathfrak{g}$. Recall that we have assumed that
 the grading on $\mathfrak{g}$ is {\em minimal} in the sense
that $\mathrm{supp}(\mathfrak{g})$ generates $Q$.

Classification of graded simple Lie algebras for which the
underlying Lie algebra $\mathfrak{g}$ is simple, reduces to
classification of gradings on simple Lie algebras. Grading on simple
Lie algebras are, to some extent,  well-studied, see \cite{EK,EK2},
and we will not study this problem in the present paper. Instead, now we
assume that $\mathfrak{g}$ is not simple.

For $\alpha\in Q$, define $\pi_{\alpha}:\mathfrak{g}\to\mathfrak{g}_{\alpha}$
as the projection with respect to the graded decomposition.
Take a non-homogeneous non-zero proper ideal $I$ of $\mathfrak{g}$. As we progress we may re-choose this $I$ later in this section.
Define the {\it size} of $I$ as
\begin{displaymath}
\mathrm{size}(I)=
\min\big\{|\mathrm{supp}(x)|\,:\,x\in I\setminus\{0\}\big\}.
\end{displaymath}

\begin{lemma}\label{Property12}
{\small\hspace{2mm}}

\begin{enumerate}[$($a$)$]
\item\label{Property12.1} For any nonzero $x\in I$, we  have
$|\mathrm{supp}(x)|>1$.
\item\label{Property12.2} We have $\pi_{\alpha}(I)=\mathfrak{g}_\alpha$, for each $\alpha\in Q$.
\end{enumerate}
\end{lemma}

\begin{proof}
Claim~\eqref{Property12.1} follows from the observation that the set
\begin{displaymath}
J=\{x\in I\,:\,|\mathrm{supp}(x)|\le 1\}
\end{displaymath}
is a nontrivial  graded ideal  of $\mathfrak{g}$ which has to be
zero as $\mathfrak{g}$ is graded simple.

To prove claim~\eqref{Property12.2}, let
$\bar I=\displaystyle\sum_{\alpha\in Q}\pi_{\alpha}(I)$. As $\mathfrak{g}_\alpha  \pi_\beta(I) \subset \pi_{\alpha+\beta}(I)$, it
follows that $\bar I$ is a nonzero graded ideal of $\mathfrak{g}$ which has to be
$\mathfrak{g}$ itself since $\mathfrak{g}$ is graded simple. This completes the proof.
\end{proof}

\subsection{Auxiliary lemmata}\label{s1.15}

From now on in this section,  we assume that $Q=Q_0\times Q_1$ 
where $Q_0$ is finite,  $\Bbbk$ contains a primitive 
root of unity of order $|Q_0|$  and any
ideal of $\mathfrak{g}$ is $Q_1$-graded. Then we have the graded
isomorphisms $\tau_f:\mathfrak{g} \to \mathfrak{g}$, for any
$f\in\hat{Q}_0$ defined as in \eqref{eq71} (with the convention that
$f(\alpha)=1$ for all $\alpha\in Q_1$).
\vspace{5mm}

\begin{lemma}\label{Property2}
{\small\hspace{2mm}}

\begin{enumerate}[$($a$)$]
\item\label{Property2.1} An ideal  $J$
of $\mathfrak{g}$ is $Q$-graded if and only if 
$\tau_f(J)\subset J$, for all
$f\in\hat{Q}_0$.
\item\label{Property2.2} The center of $\mathfrak{g}$ is zero.
\end{enumerate}
\end{lemma}

\begin{proof}
Claim~\eqref{Property2.1} is clear. Claim~\eqref{Property2.2}
follows from claim~\eqref{Property2.1} since the center is an ideal
and is invariant under all automorphisms including $\tau_f$ for 
$f\in\hat{Q}_0$.
\end{proof}

For convenience, we redefine
\begin{displaymath}
\mathrm{Inv}(I)=\{f\in \hat{Q}_0\,|\,\tau_f(I)=I\}
\end{displaymath} and set
$$P_0={\mathrm{Inv}}(I)^{\perp}:=\{\alpha\in Q_0 : f(\alpha)=1, {\text{ for all }}
f\in \mathrm{Inv}(I)\}.$$ From Remark~\ref{remabgroup} it follows that
$|P_0|=|\hat{Q}_0/\mathrm{Inv}(I)|$. Then we know that 
$\mathrm{Inv}(I)$ is a  proper subgroup of $Q_0$.

Now we can consider $\mathfrak{g}$ as a $Q/P_0$-graded Lie algebra.
The homogeneous spaces in this graded Lie algebra are indexed by
$\bar{\alpha}=\alpha+P_0$, where $\alpha\in Q$. Note that all these homogeneous
components are eigenspaces for each $\tau_f$, where $f\in \mathrm{Inv}(I)$.
For $\alpha\in Q$, we thus have
\begin{equation}\label{eqn5}
\mathfrak{g}_{\bar{\alpha}}=\bigoplus_{\beta\in P_0}\mathfrak{g}_{\alpha+\beta}.
\end{equation}
The decomposition
\begin{displaymath}
\mathfrak{g}=\bigoplus_{\bar{\alpha}\in Q/P_0}\mathfrak{g}_{\bar{\alpha}}
\end{displaymath}
is the decomposition of $\mathfrak{g}$
into a direct sum of common eigenvectors with respect to the action of all
$\tau_f$, where $f\in\mathrm{Inv}(I)$. Since $I$ is preserved by all such
$\tau_f$,  we obtain
\begin{displaymath}
I_{\bar{\alpha}}=I\cap \mathfrak{g}_{\bar{\alpha}}.
\end{displaymath}

For $\mathbf{I}:=\{\tau_f(I)\,:\, f\in \hat{Q}_0\}$, we have that
$|\mathbf{I}|=|\hat{Q}_0/\mathrm{Inv}(I)|$. 
Now we can prove the following statement.

\begin{lemma}\label{Existence}
Assume that  $\mathfrak{g}$ is not simple. Then there exists a  
non-homogeneous non-zero proper ideal $I$ of $\mathfrak{g}$ such that different 
elements in $\mathbf{I}$ have zero intersection.
\end{lemma}

\begin{proof} 
Suppose  $J\cap J'\ne 0$ for some $J\neq J'$ in $\mathbf{I}$. We may assume that  $I\cap \tau_f(I)=I_1\ne0$ for some
$f\in\hat{Q}_0\setminus \mathrm{Inv}(I).$ We see that $I_1$ is  a  non-homogeneous non-zero
proper ideal of $\mathfrak{g}$, and 
$\mathrm{Inv}(I)\subseteq \mathrm{Inv}(I_1)\subset \hat Q_0$.

Let  $r=\mathrm{size}(I)$ and $r_1=\mathrm{size}(I_1)$.
Note that $r\le |P_0|$ and $r_1\le |P_0|$.
Let $I_2$ be the subideal of $I_1$ generated by all $x\in I_1$ with
$|\mathrm{supp}(x)|=r_1$.  Take a nonzero $x\in I_1$ with
$|\mathrm{supp}(x)|=r_1$. Then  $x=\tau_f(y)$ for some $y\in I$.
If $x\not\in\Bbbk y$, then $I$ contains a linear combination of
$x$ and $y$ which has strictly smaller support. This means that $r<r_1$.
Consequently, in the case $r=r_1$, the previous argument shows  that
$\tau^{-1}_f(x)\in I_1$, for any $x\in I_1$ with
$|\mathrm{supp}(x)|=r_1$. This means that $\tau_f(I_2)\subset I_2$ and thus
either $\mathrm{Inv}(I_2)$ properly contains
$\mathrm{Inv}(I)$ or $r_1>r$.

Now we change our original ideal $I$ to $I_2$. In this way we either
increase the size of the ideal or the cardinality of the invariant 
subgroup and start all over
again. Since both the size and the cardinality of the invariant
subgroup are uniformly bounded, the process will terminate in a
finite number of steps resulting in a new ideal  of
$\mathfrak{g}$ which will have the property that different 
elements in the corresponding $\mathbf{I}$ have zero intersection.
\end{proof}

Since $Q_0$ is finite, we may further take a non-homogeneous 
non-zero proper ideal $I$ of $\mathfrak{g}$ which has the property
described in Lemma~\ref{Existence} and such that $\mathrm{Inv}(I)$ is 
minimal with respect to inclusion. Then $I$ is a $Q/P_0$-graded Lie algebra
\begin{displaymath}
I=\bigoplus_{\bar{\alpha}\in Q/P_0}I_{\bar{\alpha}}.
\end{displaymath}
From \eqref{eqn5} it follows that $\mathrm{size}(I)\le |P_0|$.
Directly from the definitions we also have $\mathrm{Inv}(I)\subset Q_0$.

\begin{lemma}\label{Property3} We have:
\begin{enumerate}[$($a$)$]

\item\label{Property3.1}  $[J, J']=0$
for any $J\neq J'$ in $\mathbf{I}$;
\item\label{Property3.2} $\displaystyle\mathfrak{g}=\bigoplus_{J\in \mathbf{I}} J$;
\item\label{Property3.3} $I$ is a simple Lie algebra.
\end{enumerate}
\end{lemma}

\begin{proof} 
Claim~\eqref{Property3.1}   follows from the fact that $[J,
J']=J\cap J'=0$.

To prove claim~\eqref{Property3.2}, we first note that
$\displaystyle\sum_{J\in \mathbf{I}} J=\mathfrak{g}$ as the left
hand side, being closed under the action of $\hat{Q}_0$, is a
non-zero homogeneous ideal of $\mathfrak{g}$ (and hence coincides with
$\mathfrak{g}$ as the latter is graded simple). Let us prove that
this sum is direct. For $J\in \mathbf{I}$, consider
\begin{displaymath}
X_J=J\cap\sum_{J'\in\mathbf{I}\setminus\{J\}}J',
\end{displaymath}
which is an ideal of  $\mathfrak{g}$. We have $[X_J,\mathfrak{g}]=0$. Hence $X_J=0$ by
Lemma~\ref{Property2}\eqref{Property2.2}. Claim~\eqref{Property3.2}
follows.

Finally,  suppose $I$ is not simple as a Lie algebra. If $[I,I]=0$,
then from  \eqref{Property3.2} we have
$[\mathfrak{g},\mathfrak{g}]=0$, which is impossible. So $[I,I]\neq 0$. We
take a non-zero proper ideal $I_1$ of $I$. From \eqref{Property3.2}  we have the direct sum
$$\bigoplus_{f\in \hat{Q}_0/\mathrm{Inv}(I)}\tau_f(I_1).$$
We see that $\mathrm{Inv}(I_1)$ is a subgroup of $\mathrm{Inv}(I)$. From the minimality of $\mathrm{Inv}(I)$ we deduce that $\mathrm{Inv}(I_1)=\mathrm{Inv}(I)$. Thus the above direct sum
is a homogeneous non-zero proper ideal of $\mathfrak{g}$
which   contradicts the graded simplicity of $\mathfrak{g}$. So  $I$ is a simple Lie algebra.
This completes the proof.
\end{proof}

Because of Lemma~\ref{Property3}, from now on we may assume that $I$ is a
 non-homogeneous non-zero proper ideal of $\mathfrak{g}$, that is a simple Lie algebra.

Let $\alpha\in Q$ and  $x\in I_{\bar{\alpha}}\setminus\{0\}$.  Since
\begin{displaymath}
 I_{\bar{\alpha}}=I\bigcap \bigoplus_{\beta\in
P_0}\mathfrak{g}_{\alpha+\beta},
\end{displaymath}
there are unique  vectors
$x_{\alpha+\beta}\in\mathfrak{g}_{\alpha+\beta}$, where $\beta\in P_0$,
such that
\begin{equation}\label{eqn9}
x=\sum_{\beta\in P_0}x_{\alpha+\beta} .
\end{equation}
Applying $\tau_f$, where $f\in \hat{Q}_0/\mathrm{Inv}(I)$, to both side of
\eqref{eqn9}, we obtain $|P_0|$ identities:
\begin{displaymath}
\sum_{\beta\in P_0}f(\alpha+\beta)x_{\alpha+\beta}=\tau_f(x), \,\,\,\,\,\,
f\in \hat{Q}_0/\mathrm{Inv}(I).
\end{displaymath}
From Lemma~\ref{Property3} \eqref{Property3.2}, we have that
$\{\tau_f(x): f\in \hat{Q}_0/\mathrm{Inv}(I)\}$ is a set of linearly
independent elements. Since $|P_0|=|\hat{Q}_0/\mathrm{Inv}(I)|$, we see that
$ \{x_{\alpha+\beta}\,:\,\beta\in P_0\}$ is a set of linearly independent elements and
\begin{displaymath}
\mathrm{span}\{\tau_f(x) \,:\, f\in \hat{Q}_0/\mathrm{Inv}(I)\}
=\mathrm{span} \{x_{\alpha+\beta}:\beta\in P_0\}.
\end{displaymath}
Thus $\{x_{\alpha+\beta}\,:\,\beta\in P_0\}$ can be uniquely determined
from the above $|Q_0|$ identities in terms of $\{\tau_f(x) \,:\, f\in
\hat{Q}_0/\mathrm{Inv}(I)\}$. Therefore the coefficient matrix
$\big(f(\alpha+\beta)\big)$, where $f\in \hat{Q}_0/\mathrm{Inv}(I)$ and
$\beta\in P_0$, is invertible. The above argument yields the following linear algebra
result:

\begin{lemma} \label{lemnn2}
Let $\alpha\in Q$ and
\begin{displaymath}
x=\sum_{\beta\in P_0}x_{\alpha+\beta}\in I_{\bar{\alpha}}\setminus\{0\},
\end{displaymath}
where $x_{\alpha+\beta}\in\mathfrak{g}_{\alpha+\beta}$, for $\beta\in
P_0$. Then $\{x_{\alpha+\beta}\,:\,\beta\in P_0\}$ is a set of linearly independent
elements and each $x_{\alpha+\beta}$ can be uniquely expressed in terms of
elements in $\{\tau_f(x) \,:\, f\in \hat{Q}_0/\mathrm{Inv}(I)\}$ and entries of the
invertible matrix $\big(f(\alpha+\beta)\big)_{f, \beta}$, where $f\in
\hat{Q}_0/\mathrm{Inv}(I)$ and $\beta\in P_0$. Consequently,
$\mathrm{size}(I)=|P_0|$.
\end{lemma}

We will also need the following recognition result.

\begin{lemma}\label{Property4}
Let $\mathfrak{g}$ and $\mathfrak{g}'$ be two  $Q$-graded
simple Lie algebras with   non-homogeneous non-zero proper
ideals $I$ and $I'$, respectively, that are simple Lie algebras. If
$\mathrm{Inv}(I)=\mathrm{Inv}(I')$ and $I\simeq I'$ as
${Q}/P_0$-graded Lie algebras, then $\mathfrak{g}$ and
$\mathfrak{g}'$ are isomorphic as $Q$-graded Lie algebras.
\end{lemma}

\begin{proof}
From the discussion above we know that both $I$ and $I'$ are simple Lie algebras.
Let $\varphi_0: I\to I'$ be an isomorphism of  $Q/P_0$-graded Lie algebras.
We know that, for each $\alpha\in Q$, we have the decomposition
\begin{displaymath}
I_{\bar\alpha}\subset
\bigoplus_{\beta\in P_0} \mathfrak{g}_{\alpha+\beta}.
\end{displaymath}

For any $\bar{f}\in \hat{Q}_0/\mathrm{Inv}(I)$,  set
\begin{displaymath}
\varphi_f:=\tau_f\circ\varphi_0\circ\tau_f^{-1}:\tau_f (I)\to \tau_f (I').
\end{displaymath}
By taking the direct sum, we obtain an isomorphism of
$Q/P_0$-graded Lie algebras as follows:
\begin{displaymath}
\Phi:=\bigoplus_{\bar{f}\in \hat{Q}_0/\mathrm{Inv}(I)}\varphi_f:
\mathfrak{g}=\displaystyle\bigoplus
_{\bar{f}\in \hat{Q}_0/\mathrm{Inv}(I)}\tau_f ( I)
\longrightarrow\mathfrak{g}'=\displaystyle\bigoplus_{\bar{f}\in
\hat{Q}_0/\mathrm{Inv}(I)}\tau_f ( I' ).
\end{displaymath}
The isomorphism $\Phi$ commutes with all $\tau_f$ by construction.
Therefore, $\Phi$ is even an isomorphism of  $Q$-graded Lie algebras.
\end{proof}

\subsection{Classification}\label{s1.175}

The following theorem is the main result of this section.

\begin{theorem}\label{Finite}
Let $Q=Q_0\times Q_1$ be an additive abelian group where $Q_0$ is finite,  
$\Bbbk$ be an arbitrary field containing a primitive root of unity of order $|Q|$,  
and $\mathfrak{g}=\displaystyle\bigoplus_{\alpha\in
Q}\mathfrak{g}_\alpha$ be a $Q$-graded  simple Lie algebra over $\Bbbk$.
Assume that any  ideal of $\mathfrak{g}$ is $Q_1$-graded. Then there exists
a subgroup $P_0\subset Q_0$ and a simple Lie algebra
${\mathfrak{a}}$ with a $Q/P_0$-grading such that
$\mathfrak{g}\simeq \mathfrak{g}(Q,P_0,{\mathfrak{a}})$.
\end{theorem}

\begin{proof}
We may assume that $\mathfrak{g}$ is not simple.
Using Lemma~\ref{Property3}, fix a   nontrivial non-graded
ideal $\mathfrak{a}$ of $\mathfrak{g}$ that is a simple Lie algebra. Then the $Q/P_0$-grading for $\mathfrak{a}$ is given by
\begin{equation} \label{eqnn3}
\mathfrak{a} =\bigoplus_{\bar{\alpha}\in Q/P_0}  \mathfrak{a}_{\bar{\alpha}},
\end{equation}
where
\begin{displaymath}
\mathfrak{a}_{\bar{\alpha}}=\mathfrak{a}\bigcap
\bigoplus_{\beta\in \alpha+P_0} \mathfrak{g}_{\beta}.
\end{displaymath}
From the definition of $P_0$ it follows that ${\mathrm{Inv}}({\mathfrak{a}})=P_0^\perp$.
Now the claim follows from Lemma~\ref{Property4} applied to
the graded Lie algebras $\mathfrak{g}$ and $\mathfrak{g}(Q,P_0,{\mathfrak{a}})$,
where in both cases the distinguished $Q/P_0$-graded ideal is $\mathfrak{a}$.
\end{proof}

The following result is a direct consequence of Theorem~\ref{Finite}.

\begin{corollary}\label{Finite2}
Let $Q$ be a finite  additive abelian group and $\mathfrak{g}$ be a
$Q$-graded  simple Lie algebra over an arbitrary field $\Bbbk$  
containing a primitive  root of unity of order $|Q|$. Then there exists a
subgroup $P\subset Q$ and a simple Lie algebra ${\mathfrak{a}}$ with
a $Q/P$-grading such that $\mathfrak{g}\simeq
\mathfrak{g}(Q,P,{\mathfrak{a}})$.
\end{corollary}

The isomorphism problem is dealt with by the following statement.

\begin{theorem}\label{iso1}
Let $Q$, $Q'$ be finite abelian groups, $\mathfrak{g}(Q,P,{\mathfrak{a}})$
be  a $Q$-graded simple Lie algebra over $\Bbbk$ such that $\Bbbk$ 
contains primitive roots of unity of orders $|Q|$ and $|Q'|$, and $\mathfrak{g}(Q',P',{\mathfrak{a}'})$
be a $Q'$-graded Lie algebra over $\Bbbk$ with minimal gradings.
Then $\mathfrak{g}(Q,P,{\mathfrak{a}})$ is graded isomorphic to
$\mathfrak{g}(Q',P',{\mathfrak{a}'})$ if and only if there is a
group isomorphism $\tau:Q\to Q'$ such that $\tau(P)=P'$ and the
simple Lie algebras $\mathfrak{a}$ and  $\mathfrak{a}'$ are graded
isomorphic.
\end{theorem}

\begin{proof} 
The direction ($\Leftarrow$) is clear,  we consider the direction ($\Rightarrow$).

Suppose $\Phi:\mathfrak{g}(Q,P,{\mathfrak{a}})\to \mathfrak{g}(Q',P',{\mathfrak{a}'})$ 
is a graded isomorphism of Lie algebras. Then there is a group isomorphism $\sigma:Q\to Q'$ 
such that $\Phi(\mathfrak{g}_\alpha)=\mathfrak{g}'_{\sigma(\alpha)}$, for all $\alpha\in Q$. 

From Lemma \ref{Property19}, we know that $\mathfrak{a}$ and  $\mathfrak{a}'$ 
are the simple subalgebras of  $\mathfrak{g}(Q,P,{\mathfrak{a}})$  and 
$\mathfrak{g}(Q',P',{\mathfrak{a}'})$ and all simple subalgebras are isomorphic. 
There is $f\in\hat Q'$ such that $\tau_f(\mathfrak{a}')=\Phi(\mathfrak{a})$. 
We have 
\begin{displaymath}
 \sigma(P)=\sigma(\text{Inv}(\mathfrak{a})^\perp)=
\text{Inv}(\Phi(\mathfrak{a}))^\perp=\text{Inv}(\tau_f(\mathfrak{a}'))^\perp=P'.
\end{displaymath}
Using restriction, one can easily see that the simple Lie algebras 
$\mathfrak{a}$ and  $\mathfrak{a}'$ 
are $Q/P$-graded isomorphic.
\end{proof}

In the above correspondence theorems, the algebras $\mathfrak{g}$ and 
${\mathfrak{a}}$  may not be central or with split centroid, which is quite 
different from the setup of Correspondence Theorem~7.1.1 in \cite{ABFP}.
Note that characterization on $Q$-graded  simple Lie algebras over a field
$\Bbbk$ for a general additive abelian group $Q$ in the
general case has to be dealt with by different methods. In what
follows, we approach this problem using a graded version of Schur's
lemma and using Correspondence Theorem~7.1.1 in \cite{ABFP}.

\section{Graded Schur's Lemma}\label{s25}

In this section we prove a graded version of Schur's lemma which we
will frequently use in the rest of the paper. This is a standard
statement, but we could not find a proper reference for the
generality we need.

Let $Q$ be an abelian
group, $\mathfrak{g}=\displaystyle\bigoplus_{\alpha\in
Q}\mathfrak{g}_\alpha$  a $Q$-graded  Lie algebra over a field
$\Bbbk$ and $W=\displaystyle\bigoplus_{{\bar\alpha}\in Q}W_{\alpha}$
a $Q$-graded  module over $\mathfrak{g}$. For $\alpha\in Q$, we call
a module homomorphism $\sigma :W\to W$ {\em homogeneous of degree $\alpha$}
provided that $\sigma(W_\beta)\subset W_{\beta+\alpha}$.

\begin{theorem}[Graded Schur's Lemma]\label{Schur}
Let $Q$ be an abelian group and
$\mathfrak{g}$ a $Q$-graded Lie algebra over an algebraically
closed field $\Bbbk$. Let $W$ be a $Q$-graded simple module over
$\mathfrak{g}$ with $\dim W< |\Bbbk|$. Then, for any fixed $\alpha\in
Q$, any two degree $\alpha$ automorphisms of $W$ differ by a scalar factor only.
\end{theorem}

\begin{proof}
Let $\mathrm{End}_0(W)$ be the algebra of all homogeneous degree
zero endomorphisms of $W$. It is enough to show that
$\mathrm{End}_0(W)=\Bbbk$.  The usual arguments give that
$\mathrm{End}_0(W)$ is a division algebra over $\Bbbk$. Then $W$,
viewed as an $\mathrm{End}_0(W)$-module, is a sum of copies of
$\mathrm{End}_0(W)$. In particular, $$\dim \mathrm{End}_0(W)\le \dim
W<|\Bbbk|.$$ Since $\Bbbk$ is algebraically closed, if
$\mathrm{End}_0(W)$ were strictly larger than $\Bbbk$, then
$\mathrm{End}_0(W)$ would contain some $\sigma$ which is
transcendental over $\Bbbk$. Then the fraction field $\Bbbk(\sigma)$
would be contained in $\mathrm{End}_0(W)$. However, we have  the
elements $\frac{1}{\sigma-a}\in \Bbbk(\sigma)$, where $a\in\Bbbk$,
which  are linearly independent. Therefore $$\dim
\mathrm{End}_0(W)\ge \dim\Bbbk(\sigma)\ge|\Bbbk|,$$ contradicting
the fact that $\dim \mathrm{End}_0(W)<|\Bbbk|$. Thus we conclude that
$\mathrm{End}_0(W)=\Bbbk$.
\end{proof}

\section{Graded  simple Lie algebras:  the case of infinite $Q$}\label{s1.3}

\begin{theorem}\label{Infinite1}
Let $Q$ be an arbitrary additive abelian group and
$\mathfrak{g}$ be a $Q$-graded  simple Lie algebra over a algebraically closed field
$\Bbbk$ such that  $\dim\mathfrak{g}<|\Bbbk|$. Then there exists a
subgroup $P\subset Q$ and a simple Lie algebra ${\mathfrak{a}}$ with
a $Q/P$-grading such that $\mathfrak{g}\simeq
\mathfrak{g}(Q,P,{\mathfrak{a}})$.
\end{theorem}

\begin{proof} 
From Theorem~\ref{Schur}, we know that $\mathfrak{g}$ is graded-central-simple 
(see Definition~4.3.1 in \cite{ABFP}). From Lemma~4.3.8 in \cite{ABFP}, 
we know that $\mathfrak{g}$ has a split centroid. The statement now follows from 
Correspondence Theorem 7.1.1 in \cite{ABFP}.
\end{proof}

Combining Lemma~\ref{Property19} with Theorems~\ref{Finite} and
\ref{Infinite1}, we obtain:

\begin{corollary}\label{cor237}
Let $Q$ be a finite additive abelian group and $\mathfrak{g}$ a $Q$-graded  simple 
Lie algebra over a field $\Bbbk$. Suppose that one of the following holds:
\begin{enumerate}[$($i$)$]
\item\label{cor237.1} $\Bbbk$ contains a primitive root of unity of order $|Q|$;
\item\label{cor237.2} $\Bbbk$ is algebraically closed and $\dim\mathfrak{g}<|\Bbbk|$.
\end{enumerate}  
Then there exists a subgroup $P\subset Q$ and a simple Lie algebra ${\mathfrak{a}}$ with
a $Q/P$-grading such that $\mathfrak{g}\simeq \mathfrak{g}(Q,P,{\mathfrak{a}})$.
\end{corollary}

The following result is a generalization of
\cite[Main~Theorem~(a)]{M} which follows directly from Theorem~\ref{Infinite1}.

\begin{corollary}\label{cor238}
Let $Q=Q_0\times Q_1$ be an additive abelian group where $Q_0$ is
torsion subgroup of $Q$. Let $\displaystyle \mathfrak{g} $ be a
finite dimensional $Q$-graded  simple Lie algebra. Then
$\mathfrak{g}$ is a $Q_0$-graded simple Lie algebra.
\end{corollary}

The necessary and sufficient conditions for two graded-simple algebras 
$\mathfrak{g}(Q,P,{\mathfrak{a}})$ where ${\mathfrak{a}}$ is central-simple
were given in Correspondence Theorem 7.1.1 in \cite{ABFP}.

From Corollary \ref{cor237},  we actually obtain a full
classification of all finite-dimensional $Q$-graded simple Lie
algebras  over any algebraically closed field of characteristic $0$
due to the recent classification of all gradings on finite
dimensional simple Lie algebras, see \cite{EK2,E,Y}. For a similar
classification over an algebraically closed field of characteristic
$p>0$ it remains only to determine all gradings on finite dimensional
simple Lie algebras. Some partial results in this direction can be found
in \cite{EK2}, see also references therein.

\section{Graded Weyl Theorem}\label{s2}

One consequence of our  Theorem~\ref{Infinite1} is that
any finite dimensional $Q$-graded simple Lie algebra over an algebraically
closed field $\Bbbk$ of characteristic $0$ is semi-simple after
forgetting the grading (note that this property is not true in positive
characteristic). This allows us to prove a graded version of the  Weyl Theorem.

\begin{theorem}[Graded Weyl Theorem]\label{Weyl}
Let $Q$ be an abelian group and
$\mathfrak{g}$ a finite dimensional $Q$-graded semi-simple Lie
algebra over an algebraically closed field $\Bbbk$ of characteristic
$0$. Then any finite dimensional $Q$-graded module $V$ over
$\mathfrak{g}$ is completely reducible as a graded module, that is, $V$ is a direct
sum of $Q$-graded simple submodules of $V$.
\end{theorem}

\begin{proof}
Since $\mathfrak{g}$
is finite-dimensional,  the minimal grading of $\mathfrak{g}$ is by
a finitely generated subgroup of $Q$. Since $\mathfrak{g}$ is finite dimensional, 
from Corollary~\ref{cor237}, we know that $\mathfrak{g}$ is $Q_0$-graded simple, 
for a finite subgroup $Q_0$ of $Q$. From Lemma \ref{Property19}, it follows that, 
as an ungraded Lie algebra, $\mathfrak{g}$ is a finite-dimensional semisimple Lie algebra.
(This can also be proved by looking at the radical of $\mathfrak{g}$ without using Corollary~\ref{cor237}).

We need to show that any $Q$-graded submodule $X$ of a $Q$-graded
finite dimensional $\mathfrak{g}$-module $W$ has a $Q$-graded
complement. By Weyl Theorem, we have an ungraded $\mathfrak{g}$-submodule $Y_1$
such that $W=X\oplus Y_1$. By \cite[Lemma~1.1]{EK} (see also
\cite[Theorem~2.3$'$]{CM}), there is a $Q$-graded submodule $Y$ of $W$ 
such that $W=X\oplus Y.$ The theorem follows.
\end{proof}

\section{Graded simple modules over graded Lie algebras}\label{5}

We will study graded simple modules over graded Lie algebras (not necessarily graded simple) in this section. 
Since we will use the graded analogue of Schur's Lemma, we will assume that
the base field $\Bbbk$ is algebraically closed later in Subsection~\ref{s4.4}.

\subsection{Construction}\label{s4.3}

Let $Q$ be an abelian group and  $P$ a subgroup of $Q$.
Let, further, $\mathfrak{g}$ be a $Q$-graded  Lie algebra over  an arbitrary field $\Bbbk$.
Note that in this section we even do not need to assume that the 
$Q$-grading of $\mathfrak{g}$ is minimal. Consider
$\mathfrak{g}=\displaystyle\bigoplus_{\bar\alpha\in
Q/P}\mathfrak{g}_{\bar\alpha}$ as a $Q/P$-graded Lie algebra with
$\mathfrak{g}_{\bar\alpha}= \displaystyle\bigoplus_{\beta\in
P}\mathfrak{g}_{\alpha+\beta}$. Let
$V=\displaystyle\bigoplus_{{\bar\alpha}\in Q/P}V_{\bar\alpha}$ be a
simple $\mathfrak{g}$-module with a fixed $Q/P$-grading.

Then we can form the  $\mathfrak{g}$-module $V\otimes \Bbbk Q$ as
follows: for $x\in \mathfrak{g}_{\alpha}$, $v\in V$ and
$\beta\in Q$, define
\begin{displaymath}
x\cdot \big(v\otimes t^{\beta})=(xv)\otimes t^{\alpha+\beta}.
\end{displaymath}
Define the  $Q$-graded $\mathfrak{g}$-module
\begin{displaymath}
M(Q,P,V):= \bigoplus_{\alpha\in Q}M(Q,P,V)_{\alpha},\quad \text{
where }\quad M(Q,P,V)_{\alpha}:=V_{\bar\alpha}\otimes t^{\alpha}.
\end{displaymath}
For example, $M(Q,Q,V)=V\otimes \Bbbk Q$ (with the
obvious $Q$-grading) while $M(Q,\{0\},V)=V$ (with the original
$Q$-grading). From the definition it follows that $\dim
M(Q,P,V)=\dim(V)|P|$, if $V$ is finite-dimensional and $P$ is finite.
These graded modules are called loop modules in \cite{EK4}.

Unlike the algebra case,  the $Q$-graded   $\mathfrak{g}$-module $M(Q,P,V)$ is generally not  graded simple.  Even it is  hard to see whether it is graded simple. We will first discuss this problem.

We say that a simple $\mathfrak{g}$-module $V$ with a $Q/P$-grading
is {\it grading extendable} if there is a subspace decomposition
\begin{displaymath}
V=\sum_{\alpha\in Q}X_{\alpha}\,\,\,\text{ with }\,\,\,
V_{\bar{\alpha}}=\sum_{{\beta}\in P}X_{{\alpha}+\beta}
\end{displaymath}
for any ${\alpha}\in Q$ (here both sums are not necessarily direct)
such that  $\mathfrak{g}_\beta X_{\alpha}\subset X_{\alpha+\beta}$
for all $\beta\in Q$ and at least one $X_{\alpha}\ne
V_{\bar{\alpha}}$. Note that the set $\{X_\alpha:\alpha\in Q\}$ 
is called a $Q$-covering of $V$ in \cite{BL}.

Now we can obtain some necessary and sufficient conditions for the
$Q$-graded module  $M(Q,P,V)$ to be $Q$-graded simple.

\begin{lemma}\label{Property21}
The $Q$-graded module  $M(Q,P,V)$ is $Q$-graded simple if and only
if the simple module $V$ is not grading extendable.
\end{lemma}

\begin{proof}
If $V$ is  grading extendable, there is a decomposition
$\displaystyle V=\sum_{\alpha\in Q}X_{\alpha}$ with $\displaystyle
V_{\bar{\alpha}}=\sum_{{\beta}\in P}X_{{\alpha}+\beta}$, for any
${\alpha}\in Q$, such that  $\mathfrak{g}_\alpha X_{\beta}\subset
X_{\alpha+\beta}$ and at least one $X_{\alpha+\beta}\ne
V_{\bar{\alpha}}$. Then the module $M(Q,P,V)$ has a nonzero proper
$Q$-graded submodule
$$\bigoplus_{\alpha\in Q}X_{\alpha}\otimes t^\alpha.
$$
Thus  $M(Q,P,V)$ is not $Q$-graded simple.

Now suppose that $M(Q,P,V)$ is not $Q$-graded simple.
Consider the ideal $I$ in
$\mathbb{C}Q$ generated by $\{t^{\alpha}-1: \alpha\in P\}$ and let
\begin{displaymath}
N=M(Q,P,V)/(M(Q,P,V))\cap(V\otimes I)).
\end{displaymath}
The module $N$ is naturally $Q/P$-graded and is, in fact, isomorphic
to the $Q/P$-graded module $V$ with the original grading. We have
that
\begin{displaymath}
N=\bigoplus_{\bar\alpha\in Q/P}V_{\bar\alpha}\otimes t^{\bar\alpha},
\end{displaymath}
where $\{t^{\bar\alpha}:\bar\alpha\in Q/P\}$ is a basis for the
group algebra of $Q/P$. Let
\begin{displaymath}
\pi: M(Q,P,V)\tto N,
\end{displaymath}
be the canonical epimorphism.

Take a proper $Q$-graded submodule $\displaystyle
X=\bigoplus_{\alpha\in Q}X_{\alpha}\otimes t^{\alpha}$ of $
M(Q,P,V)$. Since $\pi(X)=N$, we have $\displaystyle
V_{\bar{\alpha}}=\sum_{{\beta}\in P}X_{{\alpha}+\beta}$ for any
${\alpha}\in Q$ and, also, $\mathfrak{g}_\alpha X_{\beta}\subset
X_{\alpha+\beta}$. Since $W$ is proper, we have $X_\alpha\ne V_{\bar
\alpha}$.
 Thus $V$ is  grading extendable.
\end{proof}

For any $f\in \hat Q$ and any $\mathfrak{g}$-module $V$, we define a
new module $V^f$ as follows: as a vector space, we set $V^f:=V$, and the action of
$\mathfrak{g}$ on $V^f$ is given by
\begin{displaymath}
x_\alpha\circ v=f(\alpha)x_\alpha v, \,\,\text{ for all }\,\, \ x_\alpha\in  \mathfrak{g}_\alpha, v\in V.
\end{displaymath}

If $W$ is a $Q$-graded  $\mathfrak{g}$-module, for any $f\in \hat Q$,   the linear map
$$\tau_f:W\to W^f; \tau_f(w_\alpha)=f(\alpha)w_\alpha, \forall \alpha\in Q, w_\alpha\in  W_\alpha,$$
is a $Q$-graded module isomorphism since$$\tau_f(x_\alpha w)=f(\alpha)x_\alpha \tau_f(w) =x_\alpha\circ \tau_f(w) $$ for all $\alpha\in Q, x_\alpha\in \mathfrak{g}_\alpha, w\in W_\alpha.$ But $\tau_f: W\to W$ is generally not a $\mathfrak{g}$-module homomorphism from $W$ to $W$. However $\tau_f(N)$ is a submodule of $W$ for any submodule $N$ of $W$.

%

\subsection{Classification of graded simple modules: the case of finite $Q$}\label{s4.4.0}

In this section we assume that $Q$ is a finite abelian group, 
the   field $\Bbbk$ contains a primitive root of unity 
of order $|Q|$, and $\mathfrak{g}=\displaystyle\bigoplus_{\alpha\in
Q}\mathfrak{g}_\alpha$  a $Q$-graded  Lie algebra over  $\Bbbk$.
Although the major part of both results and methods for the study of 
$Q$-graded simple modules is rather different from those for the study 
of $Q$-graded simple Lie algebras in Sections~\ref{s1.15} and \ref{s1.175}, 
there are some similarities.  We will omit similar  arguments in this section.

\begin{lemma}\label{Property190}
Assume that  $Q$ is  finite and $\Bbbk$ contains a primitive  
root of unity of order $|Q|$.  Then the module $$M(Q,P,V)=\bigoplus_{f\in \hat Q/P^\perp}V^f,$$
where each summand has a naturally induced $Q/P$-grading.
\end{lemma}
 
\begin{proof}  
The proof  is similar to that of Lemma \ref{Property19}.
\end{proof}

We note that neither $P$ nor $V$ in Lemma~\ref{Property190} 
are uniquely determined in general, cf. Example~\ref{exlast2}.
 
Now let $W=\displaystyle\bigoplus_{\alpha\in Q}W_\alpha$ be a
$Q$-graded  simple $\mathfrak{g}$-module.
As before, we define  $\mathrm{supp}(v)$ for any $v=\displaystyle\sum_{\alpha\in
Q}v_\alpha\in W$, where $v_\alpha\in W_\alpha$, and
$\mathrm{size}(N)$ for any subset $N$ of the module $W$.

We assume that $W$ is not simple as a $\mathfrak{g}$-module. Let $N$
be a proper nonzero submodule of $W$. Set  $r:=\mathrm{size}(N)>1$
and define
\begin{displaymath}
R(N):=\mathrm{span}\{v\in N:\mathrm{size}(v)=r\}.
\end{displaymath}
Then $R(N)$ is a non-homogeneous non-zero proper submodule of
$W$. We will say that a submodule $N$ of $W$ is {\em
pure} of size $r$ if $\mathrm{size}(N)=r$ and, moreover, $R(N)=N$.

For $\alpha\in Q$, define $\pi_{\alpha}:W \to W_{\alpha}$
as the projection with respect to the graded decomposition.
The following lemma and its proof are similar to 
Lemmata~\ref{Property12} and \ref{Property2}.

\begin{lemma}\label{Property12-0}
{\small\hspace{2mm}}

\begin{enumerate}[$($a$)$]
\item\label{Property12-0.1} For any nonzero $x\in N$ we  have
$|\mathrm{supp}(x)|>1$.
\item\label{Property12-0.2} We have $\pi_{\alpha}(N)=W_\alpha$ for each $\alpha\in Q$.
\item A submodule $N$ of $W$ is $Q$-graded if and only if $\tau_f(N)\subset N$ for all $f\in \hat Q$.
\end{enumerate}
\end{lemma}

As we did before, we define
\begin{displaymath}
\mathrm{Inv}(N)=\{f\in \hat{Q}\,|\,\tau_f(N)=N\} ,
\end{displaymath} and set
$$P={\mathrm{Inv}}(N)^{\perp}:=\{\alpha\in Q : f(\alpha)=1, {\text{ for all }}
f\in \mathrm{Inv}(N)\}.$$ We know  that
$|P|=|\hat{Q}/\mathrm{Inv}(N)|$.
For $\mathbf{N}:=\{\tau_f(N)\,:\, f\in \hat{Q}\}$, we have that
$|\mathbf{N}|=|\hat{Q}/\mathrm{Inv}(N)|$.

Similarly to Lemma~\ref{Existence} and its proof, we have the following.

\begin{lemma}\label{Existence0}
Assume that  $W$ is a $Q$-graded simple $\mathfrak{g}$-module that 
is not simple. Then   $W$ has a   non-homogeneous non-zero proper 
submodule $N$ so that different submodules in $\mathbf{N}$ have zero intersection.
\end{lemma}

Now we have the following statement.

\begin{lemma}\label{Existence0-1}
Assume that  $W$ is a $Q$-graded simple $\mathfrak{g}$-module that  
is not simple. Then   $W$ has a   non-homogeneous simple submodule.
\end{lemma}

\begin{proof} 
Since $Q$ is finite, we may take a  non-zero proper submodule 
$N$ of $W$ having the property as in Lemma~\ref{Existence0} 
and such that $\mathrm{Inv}(N)$ is minimal with respect to inclusion and, 
further,  size$(N)$ is maximal. Then $N$ is a $Q/P$-graded submodule
\begin{displaymath}
N=\bigoplus_{\bar{\alpha}\in Q/P}N_{\bar{\alpha}}.
\end{displaymath} 

Note that size$(N)=r>1$. Take a nonzero 
$z=z_1+z_2+\cdots +z_r\in N_{\bar\alpha}$ where $z_i\in W_{\alpha+\alpha_i}, \alpha_i\in P$.
Replace our $N$ by the submodule of $W$ generated by this $z$. Then 
we still have the property of Lemma~\ref{Existence0}, 
$\mathrm{Inv}(N)$ is minimal, size$(N)$ is maximal, and, furthermore, $N$ is pure.

Since  $\mathrm{Inv}(N)$ is minimal, it follows that  
$\alpha_1,\alpha_2,\cdots,\alpha_r$ generate the group $P$. There are 
$f_1=id, f_2, \cdots, f_r\in \hat Q$ such that the $r\times r$ matrix 
$(f_i(\alpha_j))$ is invertible. Then we have the direct sum
\begin{equation}\label{WW}
W=\bigoplus_{i=1}^r\tau_{f_i}(N)
\end{equation}
since the submodule on the right hand side has size $1$.

If $N$ is not simple, we take a proper nonzero pure submodule $N'$ of $N$.  We see that $\mathrm{Inv}(N)=\mathrm{Inv}(N')$, and also
\begin{displaymath}
N'=\bigoplus_{\bar{\alpha}\in Q/P}N'_{\bar{\alpha}}.
\end{displaymath} 
Note that size$(N')\ge r$.  Since  size$(N)$ is maximal, we have  size$(N')= r$. 
We may assume that $N'$ is generated by a nonzero element 
$z'=z'_1+z'_2+\cdots +z'_r\in N_{\bar\alpha}$ where $z'_i\in W_{\alpha'+\alpha'_i}$, 
for $\alpha'_i\in P$. Then $\alpha'_1,\alpha'_2,\cdots,\alpha'_r$ also generate 
the group $P$, and the $r\times r$ matrix  $(f_i(\alpha'_j))$ is invertible.
Similarly to the above, we have the direct sum
\begin{equation}\label{WW-1}
W=\bigoplus_{i=1}^r\tau_{f_i}(N').
\end{equation}
Comparing \eqref{WW} with \eqref{WW-1}, we get
\begin{equation}\label{WW-2}
\bigoplus_{i=1}^r\tau_{f_i}(N)=\bigoplus_{i=1}^r\tau_{f_i}(N').
\end{equation}
As $N'\subsetneq N$, we have $\tau_{f_i}(N')\subsetneq \tau_{f_i}(N)$,
for each $i$, and hence \eqref{WW-2} is impossible. 
This implies that $N$ is simple.
\end{proof}

\begin{theorem}\label{FM0}
Let $Q$ be a finite  additive abelian group and $W$ be a 
$Q$-graded simple $\mathfrak{g}$-module  over $\Bbbk$ such that
$\mathrm{char}(\Bbbk)$ contains a primitive root of unity 
of order $|Q|$.  Then there exist a subgroup $P\subset Q$, 
and a   simple $\mathfrak{g}$-module $V$ with a nonextendable   
$Q/P$-grading so that $W$ is  $Q$-graded isomorphic to  $M(Q, P, V)$.
\end{theorem}

\begin{proof} From the previous lemma we know that $W$ has 
a  simple $\mathfrak{g}$-submodule $N$. We may assume that  $\mathrm{Inv}(N)$ is minimal and size$(N)=r$ is maximal.
Let $P'={\rm Inv}(N)^\perp$. Then $N$ is  $Q/P'$-graded 
\begin{equation}
N=\bigoplus_{\bar{\alpha}\in Q/P'}N_{\bar{\alpha}}.
\end{equation}

Take a nonzero
$z=z_{\alpha}+z_{\alpha+\alpha_1}+\cdots z_{\alpha+\alpha_{r-1}}\in N$ where $z_\gamma\in W_\gamma$ and $\alpha_i\in P'$. We can define a linear map
$$
\Lambda_{\alpha_i}: W\to W,\,\,\, w_\gamma\mapsto w_{\gamma+\alpha_i},
$$
where $w=w_{\gamma}+w_{\gamma+\alpha_1}+\cdots
+w_{\gamma+\alpha_{r-1}}\in U(\mathfrak{g})_{\gamma-\alpha}z$, and 
$U(\mathfrak{g})_{\gamma-\alpha}$ is the homogeneous component of grading 
$\gamma-\alpha$ of the universal enveloping algebra. The map $\Lambda_{\alpha_i}$  
is well-defined since size$(N)=r$. Further $\Lambda_{\alpha_i}$  is a bijection. 
For any $y_\gamma\in U(\mathfrak{g})_{\gamma}$, from
\begin{displaymath}
y_\gamma w=y_\gamma w_{\beta}+y_\gamma w_{\beta+\alpha_1}+
\cdots+y_\gamma w_{\beta+\alpha_{r-1}}\in N
\end{displaymath}
we have that $\Lambda_{\alpha_i}(y_\gamma
w_{\beta})=y_\gamma \Lambda_{\alpha_i}(w_{\beta})$. Thus,
$\Lambda_{\alpha_i}$ is a $Q$-graded  $\mathfrak{g}$-module automorphism
of   $W$  which is, moreover,
homogeneous of degree $\alpha_i$. 

Let $P$ be the subset consisting all $\alpha_i\in $ so that there is a graded isomorphism $\Lambda_{\alpha_i}$ of W of degree $\alpha_i$. Then $P$ is a subgroup of $P'$.

If $P\ne\{0,\alpha_1,\cdots,\alpha_{r-1}\}$, there is another $\alpha_r\in P$. Then $N'=\{x+\Lambda_{\alpha_r}(x)|x\in N\}$ is a simple $\mathfrak{g}$-submodule isomorphic to $N$. Also $\mathrm{Inv}(N)=\mathrm{Inv}(N)$  and size$(N')>r$, which contradicts the choice of $N$. Thus   
$P=\{0,\alpha_1,\cdots,\alpha_{r-1}\}$, and $N$ has the natural $Q/P$-grading
\begin{equation}\label{Lam11}
N=\bigoplus_{\bar{\bar{\alpha}}\in Q/P}N_{\bar{\bar{\alpha}}}.
\end{equation} 

Let $P^\perp=\{f\in\hat Q| f(P)=1\}$. Then
\begin{equation}W=\bigoplus_{f\in \hat Q/P^\perp}\tau_{f}(N).\end{equation}\label{WWW}

From (9), (10), (11) we see that $N_{\bar{\bar{\alpha}}}=0$ or $N_{\bar{{\alpha}}}$. For example
$N_{\bar{\bar{\alpha}}}=0$ if $\alpha\in P'\setminus P$.

From this we see that $W$ is  $Q$-graded to the submodule of $M(Q, P, N)$.
This completes the proof.
\end{proof}

We illustrate our result by the following example.

\begin{example}\label{Ex1}
{\rm
Let  $\mathfrak{g}$ be the abelian Lie algebra with basis $g_{(0,0)}, g_{(1,0)},$ $ g_{(0,1)}$ 
with the $Q=\mathbb{Z}_2\times \mathbb{Z}_2$-grading given by 
\begin{displaymath}
\mathfrak{g}_{(i,j)}={\rm{span}}\{g_{(i,j)}\}\quad
{\rm{where}}\,\,\, g_{(1,1)}=0.
\end{displaymath}
Let $W$ be the $Q$-graded simple $\mathfrak{g}$-module with basis $w_{(1,0)}, w_{(0,0)}$ 
and the action
$$g_{(i,j)}\cdot w_{(0,0)}=w_{(i, j)},\,\,\, g_{(i,j)}\cdot w_{(1,0)}=w_{(i+1, j)},$$
where $w_{(0,1)}=w_{(1,1)}=0$.
Let $V=\Bbbk v$ be the one dimensional $\mathfrak{g}$-module with the action 
$$g_{(0,0)}\cdot v=v, \,\,g_{(1,0)}\cdot v=v,\,\, g_{(0,1)}\cdot v=0.$$
We see that $W$ is isomorphic to $M(Q, P, V)$ where $P=\mathbb{Z}_2\times \{0\}$. 
In the proof of Theorem~\ref{FM0}, $P'=Q$.
}
\end{example}

\subsection{Classification of graded simple modules: the case of infinite $Q$}\label{s4.4}

Let $Q$ be an additive abelian group and
$\mathfrak{g}=\displaystyle\bigoplus_{\alpha\in
Q}\mathfrak{g}_\alpha$  a $Q$-graded  Lie algebra over  an algebraically closed field
$\Bbbk$. Let $W=\displaystyle\bigoplus_{\alpha\in Q}W_\alpha$ be a
$Q$-graded  simple $\mathfrak{g}$-module such that $\dim(W)<|\Bbbk|$.

We assume that $W$ is not simple as a $\mathfrak{g}$-module. Let $N$
be a proper nonzero submodule of $W$. Set  $r:=\mathrm{size}(N)>1$.
We may assume that $N$  is  pure of size $r$.

Fix a nonzero element
$v=v_{\beta}+v_{\beta+\alpha_1}+\cdots v_{\beta+\alpha_{r-1}}\in N$ where $v_\gamma\in W_\gamma$, and 
 $\alpha_1,\alpha_2, \cdots ,\alpha_{r-1}\in Q$. We replace our $N$ by the submodule of $W$ generated by $v$. Then
 $N$  is  still pure of size $r$. We can define a linear map
\begin{equation}\label{Lam}
\Lambda_{N, \alpha_i}: W\to W,\,\,\, w_\gamma\mapsto w_{\gamma+\alpha_i},
\end{equation}
where $w=w_{\gamma}+w_{\gamma+\alpha_1}+\cdots
+w_{\gamma+\alpha_{r-1}}\in U(\mathfrak{g})_{\gamma-\beta}v$, and $U(\mathfrak{g})_{\gamma-\beta}$ is the homogeneous component of grading $\gamma-\beta$ of the universal enveloping algebra.

For any $y_\gamma\in U(\mathfrak{g})_{\gamma}$, from
\begin{displaymath}
y_\gamma w=y_\gamma w_{\beta}+y_\gamma w_{\beta+\alpha_1}+
\cdots+y_\gamma w_{\beta+\alpha_{r-1}}\in N
\end{displaymath}
we have that $\Lambda_{N, \alpha_i}(y_\gamma
w_{\beta})=y_\gamma \Lambda_{N, \alpha_i}(w_{\beta})$. Thus,
$\Lambda_{N, \alpha_i}$ is a $Q$-graded  $\mathfrak{g}$-module automorphism
of   $W$  which is, moreover,
homogeneous of degree $\alpha_i$.  From Theorem~\ref{Schur}
it follows that $\Lambda_{N, \alpha_i}$ does not depend on the
choice of $N$ up to a scalar multiple. We thus simplify the notation
$\Lambda_{N, \alpha_i}$ to $\Lambda_{\alpha_i}$. Set
\begin{gather*}
P'=\{\alpha\in Q:
{\text{there is a degree $\alpha$ module isomorphism of }} W\},\\
D'=\mathrm{span}\{\Lambda_{ \alpha}:\alpha\in P'\}.
\end{gather*}
The following result is trivial.

\begin{lemma}\label{Property11}
{\hspace{2mm}}

\begin{enumerate}[$($a$)$]
\item\label{Property11.1} The set $P'$ is a nonzero subgroup of $Q$.
\item\label{Property11.2}  The vector space $D'$ has the structure of
an associative $P'$-graded division algebra induced by composition of
endomorphisms where $D'$ is  naturally $Q$-graded with
$\deg(\Lambda_{ \alpha})={ \alpha}$.
\end{enumerate}
\end{lemma}

Now we need the following lemma.

\begin{lemma}\label{Property123}
The $Q$-graded associative $P'$-graded division
algebra $D'$ has a maximal $Q$-graded commutative subalgebra
$D$.
\end{lemma}

\begin{proof}For any ascending chain of $Q$-graded commutative
subalgebras  of $D'$: $A_1\subset A_2, \subset \cdots\subset A_n\subset
\cdots$, the union $\cup_{i=1}^{\infty}A_i$ is also a $Q$-graded commutative
subalgebras of $D'$. 
Therefore $D'$ has a maximal $Q$-graded commutative subalgebra
$D$ by Zorn's lemma.
\end{proof}

Let $D$ be a maximal $Q$-graded commutative subalgebra of $D'$ and
let $P=\mathrm{supp}(D)$ which is a subgroup of $Q$. We remark that, 
unlike the algebra case, here, in general, $D\ne D'$, see Examples~\ref{exlast2}.

From Lemma~1.2.9(i) in \cite{P}, we have the following:

\begin{lemma}\label{GroupAlg}
The $P$-graded associative division algebra $D$
is isomorphic to the group algebra $\Bbbk P$.
\end{lemma}

Now, we have   elements  $\Lambda_{\alpha}$,
for $\alpha\in P$,  satisfying
\begin{displaymath}
\Lambda_{\alpha}\Lambda_{\beta}=\Lambda_{\alpha+\beta}\,\,\,
\text{ for all }\,\,\,{\alpha},\beta\in
P.
\end{displaymath}

Let $\boldsymbol{\Lambda}=\{\Lambda_\alpha : \alpha\in P\}$ and 
$V'$ be the span of the set
\begin{displaymath}
\{v_{\beta}-\Lambda_{\alpha}(v_{\beta}):\beta\in Q,v_{\beta}\in
W_{\beta} \text{ and }\alpha\in P\}.
\end{displaymath}
Then $V'$ is a submodule of $W$.

\begin{lemma}\label{Property13} 
The submodule $V'$ of $W$ is a proper submodule of $W$.
\end{lemma}

\begin{proof}
Let $\{\beta_j : j \in B\}$ be a set of representatives  for cosets
in $Q/P$, where $B$ is an index set. Let $\{x_{\beta_j}^{(k)}: k\in
B_j\}$ be a basis of $W_{\beta_j}$, where $B_j$ is an
index set. Then the set
\begin{displaymath}
\{\boldsymbol{\Lambda}(x^{(k)}_{\beta_j})\,:\, j\in B, k\in B_j\}
\end{displaymath}
is a basis for $W$. For $j\in
B$ and $k\in B_j$, set
\begin{displaymath}
S_{jk}:=\{x^{(k)}_{\beta_j}- \Lambda(x^{(k)}_{\beta_j})\,:\,
\Lambda\in \boldsymbol{\Lambda}\}
\end{displaymath}
and
\begin{displaymath}
S=\bigcup_{j\in B}\bigcup_{k\in B_j} S_{jk}.
\end{displaymath}
Comparing supports of involved elements, it is easy to see that the sum
\begin{displaymath}
\mathcal{I}=\sum_{j\in B}\sum_{k\in B_j}\mathrm{span}(S_{jk})
\end{displaymath}
is direct. Define the linear map
\begin{equation}\label{eqapr21}
\sigma:W\to \Bbbk
\text{ such that }\Lambda(x^{(k)}_{\beta_j})\mapsto 1\text{ for all }
\Lambda\in \boldsymbol{\Lambda}.
\end{equation}
Clearly $\sigma$ is onto and $\sigma(V')=0$.  Therefore
$V'\neq W$ and thus $V'$ is a proper
submodule of $W$.  By construction, this submodule  is pure of size two.
\end{proof}

\begin{lemma}\label{Property26}
The submodule $V'$ is a  proper maximal
submodule of $W$.
\end{lemma}

\begin{proof} Clearly, $V'$ is a pure submodule of $W$  of size $2$.
Assume that $V'$ is not maximal.
Then there is a proper submodule $V''$ of $W$ which properly
contains $V'$. This means that there exist different elements
$\beta_0,\beta_0+\gamma_1,\beta_0+\gamma_2,\dots, \beta_0+\gamma_r$ in $Q$ and non-zero
elements
\begin{displaymath}
x_{\beta_0}\in \mathfrak{g}_{\beta_0}, x_{\beta_0+\gamma_1}\in
\mathfrak{g}_{\beta_0+\gamma_1}, x_{\beta_0+\gamma_2}\in
\mathfrak{g}_{\beta_0+\gamma_2},\dots, x_{\beta_0+\gamma_r}\in
\mathfrak{g}_{\beta_0+\gamma_r}
\end{displaymath}
such that
\begin{equation}\label{eqnn75}
x:=x_{\beta_0}+x_{\beta_0+\gamma_1}+x_{\beta_0+\gamma_2}+\dots+x_{\beta_0+\gamma_r}\in
V''\setminus
V'.
\end{equation}
Assume that $r$ is minimal possible for elements in $V''\setminus V'$. 
Note that $r>0$ as $V''\neq W$.
If any $\gamma_i$ or any difference $\gamma_j-\gamma_i$ for $i\neq j$ is in $P$,
then we can use definition of $V''\setminus V'$ and subtract from $x$ an element in
$V'$ getting a new element in $V''\setminus V'$ with
strictly smaller support. Therefore neither $\gamma_i$ nor any difference
$\gamma_j-\gamma_i$ for $i\neq j$ is in $P$. In particular, it follows that
$V''\setminus V'$ does not contain any element whose support would be a proper
subset of the support of $x$.
Note that any partial sum of $x$ in \eqref{eqnn75}
is not in $V''\setminus V'$ for otherwise either this sum or its complement to
$x$ would be an element in $V''\setminus V'$ with strictly smaller support.

Now we fix all these $\gamma_i$. 
Next we claim that, for a fixed $x_\beta\in W_{\beta}$
for which an element $x$ of the form \eqref{eqnn75} in $V''\setminus V'$
exists (replacing $\beta_0$ with $\beta$), the elements
\begin{displaymath}
x_{\beta+\gamma_1}\in
\mathfrak{g}_{\beta+\gamma_1}, x_{\beta+\gamma_2}\in
\mathfrak{g}_{\beta+\gamma_2},\dots, x_{\beta+\gamma_r}\in
\mathfrak{g}_{\beta+\gamma_r}
\end{displaymath}
leading to such $x$ are uniquely determined. Indeed, otherwise the non-zero
difference between two such
elements of the form \eqref{eqnn75} would have a strictly smaller support and hence
would belong to $V'$ because of the minimality of $r$.
This is, however, impossible by the previous paragraph.

By the above arguments,  we have the
$\mathfrak{g}$-module isomorphisms $\Lambda_{\gamma_i}: W\to W$
mapping $x_{\beta}$ to $x_{\beta+\gamma_i}$ defined as in  \eqref{Lam}.
We denote by $V_1$ the span of all elements of  the form \eqref{eqnn75}
where $\beta\in Q$ is arbitrary. From this construction and the above properties
it follows that $V_1$ is a pure submodule of size $r+1$ that is contained in $V''$.

Now each $x$ of the form \eqref{eqnn75} can be uniquely written as
\begin{displaymath}
x=x_\beta+\Lambda_{\gamma_1}(x_{\beta})+\dots+\Lambda_{\gamma_r}(x_{\beta})
\end{displaymath}
and all $\gamma_i$ are in $P'$ defined before Lemma \ref{Property11}.
For any ${\alpha}\in P$,  we have $x-\Lambda_{\alpha}(x)\in V'$, 
$\Lambda_{\alpha}(x)=x-(x-\Lambda_{\alpha}(x))\in V''\setminus  V'$ and hence
\begin{displaymath}
\Lambda_{\alpha}(x) =\Lambda_{\alpha}(x_\beta)+
\Lambda_{\alpha}\Lambda_{\gamma_1}(x_{\beta})+\dots
+\Lambda_{\alpha}\Lambda_{\gamma_2}(x_{\beta})\in
V_1.
\end{displaymath}
At the same time, we have
\begin{displaymath}
y=\Lambda_{\alpha}(x_\beta)+\Lambda_{\gamma_1}\Lambda_{\alpha}(x_{\beta})+
\dots+\Lambda_{\gamma_r}\Lambda_{\alpha}(x_{\beta})\in
V_1,
\end{displaymath}
as this element has the form \eqref{eqnn75} with $x_\beta$ replaced by
$\Lambda_{\alpha}(x_\beta)$.

We have $\Lambda_{\alpha}(x)-y\in V_1$. If we would have $\Lambda_{\alpha}(x)-y\neq 0$,
then $\Lambda_{\alpha_j}(x)-y\not\in V_1$ since $V_1$
does not contain elements with such support. Hence $\Lambda_{\alpha}(x)-y=0$, i.e., 
$\Lambda_{\alpha}\Lambda_{\gamma_i}=\Lambda_{\gamma_i}\Lambda_{\alpha}$ for all
$\alpha\in P$ and all $i=1,2,\dots,n$. Since $D$ is the maximal graded 
commutative subalgebra of $D'$, thus at least one $\gamma_i\in P$, a contradiction.
The claim follows.
\end{proof}

%

Now we have the following:

\begin{theorem}\label{Mod2}
Let $Q$ be an
 additive abelian group and
$\mathfrak{g}$ be a $Q$-graded Lie algebra over  an algebraically closed field $\Bbbk$. Let
$W$ be a graded simple $\mathfrak{g}$-module such that
$\dim(W)<|\Bbbk|$. Then there is a subgroup $P\subset Q$ and a
simple ${\mathfrak{g}}$-module $V$ with a   $Q/P$-grading such
that  $W$ is $Q$-graded isomorphic to  $M(Q,P,V)$.
\end{theorem}

\begin{proof}
From Lemma \ref{Property26} we have that the module $V=W/V'$
is a simple  $\mathfrak{g}$-module with a $Q/P$-grading. It is easy
to verify  that the $Q$-graded canonical map
\begin{displaymath}
\begin{array}{rcl}
W&\to & M(Q,P,V),\\
v_{\alpha}&\mapsto &(v_{\alpha}+V')\otimes t^{\alpha}
\end{array}
\end{displaymath}
is a degree $0$ injective homomorphism of $\mathfrak{g}$-modules. Thus $W$ is $Q$-graded isomorphic to  $M(Q,P,V)$.
\end{proof}

We note that a special case of Theorem~\ref{Mod2} was obtained in \cite{EK}
with a totally different approach. We remark that neither $P$ nor $V$ in
Theorem~\ref{Mod2} are uniquely determined. Further,
$M(Q,P,V)$ might have non-trivial graded automorphisms, see 
Lemma~\ref{Property134'} below.

Now we want to consider graded isomorphisms between $\mathfrak{g}$-modules
of the form $M(Q,P,V)$. By construction, we have $M(Q,P,V)^f= M(Q,P,V^f)$
for any $f\in \hat Q$. Note that $V$ and $V^f$ are not isomorphic as
$\mathfrak{g}$-modules in general. Our next observation is the following.

\begin{lemma}\label{lem0617}
In the situation above,  there is a degree zero isomorphism of graded
$\mathfrak{g}$-modules  $M(Q,P,V)$ and $M(Q,P,V)^f$.
\end{lemma}

\begin{proof}
Define $\Psi:M(Q,P,V)^f\to M(Q,P,V)$ by sending $v\mapsto \frac{1}{f(\alpha)}v$ for any
$\alpha\in Q$ and $v\in M(Q,P,V)_{\alpha}$. Then, using the definitions,
for all $\alpha,\beta \in Q$, $x_{\beta}\in \mathfrak{g}_{\beta}$ and
$v\in M(Q,P,V)_{\alpha}$, we have
\begin{displaymath}
\Psi(x_{\beta}\circ v) =\Psi(f(\beta)x_{\beta} v)=\frac{f(\beta)}{f(\alpha+\beta)}x_{\beta} v=
\frac{1}{f(\alpha)}x_{\beta} v=x_{\beta}\Psi(v).
\end{displaymath}
As $\Psi$ is obviously bijective, the claim follows.
\end{proof}

\begin{lemma}\label{Property134'}
Let $\mathfrak{g}$ be a $Q$-graded Lie algebra and $P$ subgroup of $Q$.
Let, further, $V$  be a simple $\mathfrak{g}$-module of dimension
smaller than $|\Bbbk|$ with a nonextendable  grading over $Q/P$.
Assume $\alpha\in Q$. Then there is  a degree $\alpha$ graded automorphism
$\tau: M(Q,P,V) \to M(Q,P,V)$ if and only if there is  $f\in\hat{Q}$
and  a degree $\bar\alpha\in Q/P$
isomorphism $\mu: V^f\to V$ of $Q/P$-graded $\mathfrak{g}$-modules.
\end{lemma}

\begin{proof}
To prove necessity, write $\tau$ in the form
\begin{displaymath}
 \begin{array}{rcl}\tau:
M(Q,P,V)& \to& M(Q,P,V),\\
v_{\bar\beta}\otimes t^\beta &\mapsto& \mu_{\beta}(v_{\bar\beta})\otimes t^{\beta+\alpha},
\end{array}
\end{displaymath}
where $\beta\in Q$ and $\mu_{\beta}:V_{\bar\beta}\to V_{\bar\beta+\bar\alpha}$ is a
linear isomorphism of vector spaces. For any ${\kappa}\in P$, the map
$\tau\circ \Lambda_{\kappa}\circ\tau^{-1}: M(Q,P,V) \to M(Q,P,V))$
is a  $\mathfrak{g}$-module
automorphism of degree ${\kappa}$. By Graded Schur's Lemma,
there exists $a_{\kappa}\in \mathbb{C}^*$ such that
$\tau\circ \Lambda_{\kappa}\circ\tau^{-1}=a_{\kappa} \Lambda_{\kappa}$,
that is, $\tau\circ \Lambda_{\kappa}=a_{\kappa} \Lambda_{\kappa}\circ\tau$.
This implies that $\mu_{\beta+{\kappa}}=a_{\kappa}\mu_\beta$ for all $\beta\in Q$ and ${\kappa}\in P$
(note that $a_{\kappa}$ does not depend on $\beta$).

Since $\mathbb{C}^*$ is a divisible and hence injective abelian group,
we can extend the map $\kappa\to a_{\kappa}$ to a character $f$ of $Q$.
Consider the isomorphism $\Psi$ constructed in the previous lemma.
We have the graded $\mathfrak{g}$-module isomorphism
\begin{displaymath}
\begin{array}{rcl}\tau\Psi:M(Q,P,V^f)& \to& M(Q,P,V),\\
v_{\bar\beta}\otimes t^\beta &\mapsto& f(-\beta)
\mu_{\beta}(v_{\bar\beta})\otimes t^{\beta+\alpha}.
\end{array}
\end{displaymath}
From the definitions it follows that for all $\beta\in Q$ and $\kappa\in P$ we have
\begin{displaymath}
f(-\beta-\kappa)\mu_{\beta+{\kappa}}=f(-\beta)\mu_\beta.
\end{displaymath}
Then we have a vector space automorphism $\mu: V^f\to V$ of degree $\bar\alpha$ given by $\mu(v_{\bar\beta})=f(-\beta)\mu_{\beta}(v_{\bar\beta})$, for every
$v_{\bar\beta}\in V_{\bar\beta} \ (=V^f_{\bar\beta})$.
For any $x_{\gamma}\in\mathfrak{g}_{\gamma}$ and $v_{\bar{\beta}}\in V_{\bar{\beta}}$, we have
\begin{displaymath}
\begin{array}{rcl}
\mu(x_{{\gamma}}\circ
v_{\bar{\beta}} )\otimes
t^{\gamma+\beta+\alpha}&=&\tau\Psi(x_{\gamma}\circ
 v_{\overline{\beta}}\otimes t^{\beta+\gamma})\\ &=&\tau\Psi(x_{\gamma}\circ(
 v_{{\bar\beta}}\otimes t^{\beta}))\\
&=&x_{{\gamma}}\tau\Psi( v_{{\bar\beta}}\otimes t^{\beta})\\
&=&x_{{\gamma}}( \mu( v_{\bar\beta})\otimes t^{\beta+\alpha})
\\ &=&x_{{\gamma}} \mu( v_{\bar\beta})\otimes
t^{\gamma+\beta+\alpha}.
\end{array}
\end{displaymath}
Therefore $\mu(x_{{\gamma}}\circ v_{\bar\beta})=x_{{\gamma}} \mu( v_{\bar\beta})$,
and thus $\mu$ is an isomorphism of $Q/P$-graded $\mathfrak{g}$-modules
and it has degree $\bar\alpha\in Q/P$ by construction.

Now let us prove sufficiency.
Suppose $f\in\hat Q$ and  $\mu: V^f \to V$ is a degree $\bar\alpha\in Q/P$
isomorphism of $Q/P$-graded $\mathfrak{g}$-modules. Then
\begin{displaymath}
f(\gamma)\mu(x_{{\gamma}}  v_{\bar\beta})=\mu(x_{{\gamma}}\circ  v_{\bar\beta})=
x_{{\gamma}} \mu( v_{\bar\beta})
\end{displaymath}
and this implies that
\begin{displaymath}
 \begin{array}{rcl}\tau:
M(Q,P,V)& \to& M(Q,P,V),\\
v_{\bar\beta}\otimes t^\beta &\mapsto& f({\beta})\mu(v_{\bar\beta})\otimes t^{\beta+\alpha},
\end{array}
\end{displaymath}
where $\beta\in Q$ and $v_{\bar\beta}\in V_{\bar\beta}$, is a degree $\alpha$  automorphism
of the $Q$-graded $\mathfrak{g}$-module $M(Q,P,V)$.
\end{proof}

The following example, suggested by A.~Elduque and M.~Kochetov,
shows that there are different choices for $P$ in our Theorem~\ref{Mod2}.

\begin{example}\label{exlast}
{\rm Consider the Lie algebra $\mathfrak{sl}_2(\mathbb{C})$ with the standard basis
\begin{displaymath}
h=\left(\begin{matrix}1& 0\\ 0& -1 \end{matrix}\right),
\ e=\left(\begin{matrix}0& 1\\ 0& 0 \end{matrix}\right),
\ f=\left(\begin{matrix}0& 0\\ 1& 0 \end{matrix}\right).
\end{displaymath}
Let  $Q=\mathbb{Z}_2 \times \mathbb{Z}_2$. Define on $\mathfrak{sl}_2$
the following $Q$-grading: the element $h$ has degree $(1,0)$,
the element $e+f$ has degree $(0,1)$ and the element $e-f$ has degree $(1,1)$.
This grading can be extended to the algebra of all $2\times 2$ matrices by setting
the degree of the identity matrix to be $(0,0)$.
Let $W$ be the vector space of all  $2\times 2$ matrices considered as an
$\mathfrak{sl}_2$-module via left multiplication. Then $W$ becomes a
$Q$-graded simple $\mathfrak{sl}_2$-module. Now one can check that
$W=M(Q,P, V)$ for the following three choices of $P$ and a fine grading on the
natural $\mathfrak{sl}_2$-module $V=\mathbb{C}^2$:
\begin{gather*}
P=\{0\}\times \mathbb{Z}_2,\quad  \deg\left(\begin{matrix}1\\ 1
\end{matrix}\right)=(0,0) \text{ and } \deg\left(\begin{matrix}1\\ -1
\end{matrix}\right)=(1,0),\\
P= \mathbb{Z}_2\times\{0\},\quad \deg\left(\begin{matrix}1\\ 0
\end{matrix}\right)=(0,0) \text{ and } \deg\left(\begin{matrix}0\\ 1
\end{matrix}\right)=(0,1),\\
P=\{(0,0), (1,1)\},\quad \deg\left(\begin{matrix}\mathbf{i}
\\ 1  \end{matrix}\right)=(0,0) \text{ and } \deg\left(\begin{matrix}1 \\
\mathbf{i}  \end{matrix}\right)=(1,0),
\end{gather*}
where $\mathbf{i}^2=-1$ and $(1,0)=(0,1)$ in $Q/P$ in the last case.
}
\end{example}

Here is an interesting example with different flavor in the case when
$\mathfrak{g}$ is not simple.

\begin{example}\label{exlast2}
{\rm
Consider $\mathfrak{g}=\mathfrak{sl}_2(\mathbb{C})\oplus \mathfrak{sl}_2(\mathbb{C})$ with the natural
$\mathbb{Z}_2$-grading given by
\begin{displaymath}
\mathfrak{g}_{\overline{0}}=\{(x,x)\,\vert\, x\in  \mathfrak{sl}_2(\mathbb{C})\} \quad\text{ and }\quad
\mathfrak{g}_{\overline{1}}=\{(x,-x)\,\vert\, x\in  \mathfrak{sl}_2(\mathbb{C})\}.
\end{displaymath}
We fix the standard triangular decomposition in each $\mathfrak{sl}_2(\mathbb{C})$ and consider highest
weight $\mathfrak{g}$-modules $L(h_1,h_2)$ with respect to this decomposition, where $h_i$, $i=1,2$,
gives the highest weight for the $i$-th copy of $\mathfrak{sl}_2(\mathbb{C})$. Then $L(h_1,h_2)$
admits a $\mathbb{Z}_2$-grading if and only if $h_1=h_2$, in which case it is isomorphic to
$M(\mathbb{Z}_2,\{0\},L(h_1,h_2))$. If $h_1\neq h_2$, then $L(h_1,h_2)\oplus L(h_2,h_1)$ has the natural
$\mathbb{Z}_2$-grading making it a graded simple $\mathfrak{g}$-module isomorphic to both
$M(\mathbb{Z}_2,\mathbb{Z}_2,L(h_1,h_2))$ and $M(\mathbb{Z}_2,\mathbb{Z}_2,L(h_2,h_1))$.\qed
}
\end{example}

Our Classification Theorem~\ref{Mod2} reduces
construction of $Q$-graded simple modules over a $Q$-graded Lie
algebra $\mathfrak{g}$ to classification of non-extendable
$Q/P$-grading on all simple $\mathfrak{g}$-modules, for any subgroup
$P$ of $Q$. Some results in this direction can be found in \cite{EK}.
We note that  \cite[Theorem~8]{EK} determined the necessary and sufficient
conditions for two finite dimensional $M(Q,P,V)$ to be graded isomorphic
if $\mathfrak{g}$ is a finite dimensional simple Lie algebra.
We complete the paper with the following question:

\begin{problem}\label{openproblem}
{\rm
Find necessary and sufficient conditions for two graded simple $\mathfrak{g}$-modules
$M(Q,P,V)$ and $M(Q,P',V')$ to be isomorphic.
}
\end{problem}

Further studies on graded simple modules are going on in the recent paper \cite{EK4}.

\vspace{.2cm}

\begin{center}
\bf Acknowledgments
\end{center}

\noindent The research presented in this paper was carried out
during the visit of both authors to the Institute Mittag-Leffler.
V.M. is partially supported by the Swedish Research Council,
Knut and Alice Wallenbergs Stiftelse and the Royal Swedish Academy of Sciences.
K.Z. is partially supported by NSF of China (Grant 11271109) and NSERC.
We thank Alberto Elduque and Mikhail Kochetov for comments on the original version of
the paper, in particular, for pointing out two subtle mistakes.

\vspace{0.2cm}

\noindent  V.M.: Department of Mathematics, Uppsala University, Box
480, SE-751 06, Uppsala, Sweden; e-mail:
{\tt mazor\symbol{64}math.uu.se} \vspace{0.2cm}

\noindent K.Z.: Department of Mathematics, Wilfrid Laurier
University, Waterloo, Ontario, N2L 3C5, Canada; and College of
Mathematics and Information Science, Hebei Normal (Teachers)
University, Shijiazhuang 050016, Hebei, P. R. China; e-mail:
{\tt kzhao\symbol{64}wlu.ca}
\end{document}